\documentclass[reqno]{amsart}

\usepackage{graphicx, mathscinet}

\newtheorem{theorem}{Theorem}

\newtheorem{proposition}[theorem]{Proposition}

\newtheorem{corollary}[theorem]{Corollary}
\newtheorem{lemma}[theorem]{Lemma}

\newtheorem{hypothesis}[theorem]{Hypothesis}

\newtheorem{remark}{Remark}

\def\qed{\hbox{${\vcenter{\vbox{		 
   \hrule height 0.4pt\hbox{\vrule width 0.4pt height 6pt
   \kern5pt\vrule width 0.4pt}\hrule height 0.4pt}}}$}}

\def\cD{\mathcal D}

\def\cF{\mathcal F}

\def\cL{\mathcal L}
\def\cM{\mathcal M}
\def\cN{\mathcal N}

\def\bC{\mathbb C}

\def\bE{\mathbb E}

\def\bN{\mathbb N}

\def\bP{\mathbb P}

\def\bR{\mathbb R}

\def\bZ{\mathbb Z}

\begin{document}

\title{High Order Heat-type Equations and Random Walks on the Complex Plane}

\author{Stefano Bonaccorsi and Sonia Mazzucchi} 

\address{Stefano Bonaccorsi \newline
Dipartimento di Matematica, Universit\`a di
    Trento, via Sommarive 14, 38123 Povo (Trento), Italia
}
\email{stefano.bonaccorsi@unitn.it}

\address{Sonia Mazzucchi \newline
Dipartimento di Matematica, Universit\`a di
    Trento, via Sommarive 14, 38123 Povo (Trento), Italia
}
\email{ sonia.mazzucchi@unitn.it}


\begin{abstract}
A probabilistic construction for the solution of a general class of high order heat-type equations is constructed in terms of the scaling limit of  random walks in the complex plane. 
\end{abstract}

\maketitle


\section{Introduction}\label{sez1}

The connection between the solution of parabolic equations associated to second-order elliptic operators and the theory of stochastic processes is a largely studied topic \cite{Dynkin2006}.
The main instance is the {\em Feynman-Kac formula} \eqref{Fey-Kac}, providing a representation of the solution of the heat equation \eqref{heat} with possibly a potential $V\in C_0^\infty (\bR^d)$
\begin{equation} \label{heat}\left\{ \begin{array}{l}
\frac{\partial}{\partial t}u(t,x)=\frac{1}{2}\Delta u(t,x) -V(x)u(t,x),\qquad t\in \bR^+, x\in\bR^d\\
u (0,x)=u _0(x)\\
\end{array}\right. \end{equation}
in terms of an integral with respect to the measure of the Wiener process, the mathematical model of the Brownian motion \cite{Karatzas1991}:
\begin{equation}\label{Fey-Kac}
u (t,x)=\bE^x[ e^{-\int_0^tV(\omega(s))ds}u_0( \omega(t))].\end{equation}
If the Laplacian in Eq. \eqref{heat} is replaced by an higher order differential operator, i.e. if we consider for instance a Cauchy problem of the form
\begin{equation} \label{heatN}
\left\{ \begin{array}{l}
\frac{\partial}{\partial t}u(t,x)=(-1)^{M+1} \Delta^M u(t,x) -V(x)u(t,x),\quad t\in \bR^+, x\in\bR, \\
u (0,x)=u _0(x)\\
\end{array}\right. \end{equation}
where  $M>1$ is an integer, 
then a formula analogous to \eqref{Fey-Kac},  giving the solution of \eqref{heatN} in terms of the expectation with respect to the measure associated to a Markov process, is lacking. In fact, such a formula cannot be proved for semigroups whose generator does not satisfy the maximum principle, as in the case of $ \Delta^M$ with $M>1$ \cite{Sinestrari1976}.

One of the reasons making  the higher powers of the Laplacian and equation \eqref{heatN} more difficult to handle than the traditional heat equation is the fact that, unlike the case where $M=1$, for $M>1$ the fundamental solution $G_t(x,y)$, $t\in\bR^+$, $x,y\in\bR$, is not positive.  
In fact it has an oscillatory behavior, changing sign an infinite number of times \cite{Hochberg1978}.
Consequently it cannot be interpreted as the density of a positive probability measure, as the Gaussian transition densities of the Brownian motion, but only as the density of a signed measure. 
 This fact has the troublesome consequence that if one uses $G$ as a signed transition probability density in the construction of a generalized stochastic process with real path and independent increments,  the resulting measure  on $\bR^{[0,+\infty)} $ would have infinite total variation. 
 This fact was pointed out in \cite{Krylov1960}   and can be regarded as a particular case of a general result by E. Thomas \cite{Thomas2001}, generalizing Kolmogorov existence theorem to projective families of signed or complex measures instead of probability measures. 
 In other words it is not possible to find a stochastic process $X_t$ which plays for the parabolic equation \eqref{e1} the same role that the Wiener process plays for the heat equation.

We would like to point out that the  problem of the probabilistic representation of the solution of the Cauchy problem \eqref{heatN} presents some similarities with the problem of the mathematical definition of Feynman path integrals and the functional integral representation for the solution of the Schr\"odinger equation (see \cite{ Mazzucchi2009, Johnson2000} for a discussion of this topic). 
Indeed in both cases it is not possible to implement an integration theory of Lebesgue type in terms of a bounded variation measure on a space of real  paths \cite{Cameron1960/1961}.  
This means that an analogous of the Feynman-Kac formula for the parabolic equation \eqref{heatN}, namely a representation for its solution  of the form:
 \begin{equation}\label{Fey-KacN}
u (t,x)=\int_{\omega(0)=x} e^{-\int_0^tV(\omega(s))ds}u_0( \omega(t))d\mathbb P_M(\omega),
\end{equation}
(where $\mathbb P_M$ should be some ``measure'' on a space of ``paths'' $\omega:[0,t]\to\bR$) cannot be realized in  terms of a well defined Lebesgue-type integral, but, in a weaker sense, in terms of a linear functional on a suitable class of "integrable functions",  under some restrictions on the initial datum $u_0$ and  the potential $V$. An analogous approach has been successfully implemented in the case of the Schr\"odinger equation \cite{Albeverio2008}.

Several attempts have been made to relate such problem, as in the case $M = 1$, to a random process,  in particular
for the case $M = 2$ (known in the literature as the {\em biharmonic} operator). 

One of the first approaches was introduced by  Krylov \cite{Krylov1960} and continued by Hochberg \cite{Hochberg1978}, who introduced a stochastic pseudo-process whose transition probability function is not positive definite and realized formula \eqref{Fey-KacN} in terms of the expectation with respect to a  signed measure on $\bR^{[0,t]}$ with infinite total variation. 
For this reason the integral in  is not defined in Lebesgue sense, but  is meant as limit of finite dimensional cylindrical approximations \cite{Beghin2000}.
It is worthwhile to mention that an analogous of the arc-sine law \cite{Hochberg1978, Hochberg1994}, of the  central limit theorem \cite{Hochberg1980} and of It\^o formula and It\^o stochastic calculus \cite{Hochberg1978} have been proved for the Krylov-Hochberg pseudo-process.
In the same line, we shall mention the papers by  Nishioka \cite{Nishioka1996, Nishioka2001}.
We also mention the work by D. Levin and T. Lyons \cite{Levin2009} on rough paths, conjecturing that the signed measure (with infinite total variation)  associated to the pseudo-process could exist on the quotient space of equivalence classes of paths corresponding to different parametrization of the same path.

A different approach was proposed by Funaki \cite{Funaki1979} and continued by Burdzy \cite{Burdzy1993}, based 
on  a complex valued stochastic process with dependent increments constructed by composing two independent Brownian motions. 
Formula \eqref{Fey-KacN}, representing the solution of \eqref{heatN}  in the case $M=2$ and $V=0$,  can be realized as an integral  with respect to a well defined positive probability measure on a complex space, at least  for a suitable class of analytic initial datum $u_0$.
The result can be generalized to  partial differential equations of order $2^n$, by multiple iterations of suitable processes \cite{Funaki1979, Hochberg1996, Orsingher1999}. 
These results are also related to Bochner subordination \cite{Bochner1955}. There are also similarities between the Funaki's process  and the {\em iterated Brownian motion} \cite{Burdzy1993}, but the latter is not connected to the  probabilistic representation of the solution of a partial differential equation with regular coefficients. 
In fact the processes constructed by iterating copies of independent BMs (or other process) are associated to higher order PDE of particular form, where the initial datum plays a particular role and enters also in the differential equation \cite{Allouba2002}.

Complex valued processes, connected to PDE of the form \eqref{heatN} have been also proposed by other authors by means of different techniques \cite{Burdzy1995, Madrecki1993, Burdzy1996}. 
 It is worthwhile also to mention a completely different approach proposed by R. L\'eandre \cite{Leandre2010}, which has some analogies with the mathematical realization of Feynman path integrals by means of white noise calculus \cite{Hida1993}. 
 Indeed  L\'eandre has recently constructed a "probabilistic representation" of the solution of the Cauchy problem \eqref{heatN} not as an integral with respect to a
 measure but as an infinite dimensional distribution on the Connes space \cite{Leandre2003, Leandre2006}.
\\
We eventually mention another probabilistic approach to the equation $\Delta^k u=0$ described in \cite{Helms1967}.

\smallskip

In the present paper we propose a new probabilistic construction for the solution of a general class of high order heat-type equations.
\par\noindent
Let us fix an integer $N\in \bN$, with $N > 2$,  and consider the parabolic Cauchy problem
\begin{equation}
\label{e1}
\left\{ \begin{array}{l}
\partial_t u(t,x) = \alpha \,  \partial_x^N u(t,x),
\\
u(0,x) = f(x), \qquad x \in \bR,
\end{array}\right. \end{equation}
where $\alpha \in\bC$ is a complex constant and $f:\bR\to\bC$ the initial datum. It is worthwhile to point out that, extending the problem in \eqref{heatN}, 
we can handle the case of general (even and odd) powers of the  operator $\partial_x$ appearing on the right hand side of Eq. \eqref{e1}. 
Moreover we do not only study the problem on the real line, but we can also consider (bounded) domains with  different boundary conditions.

In sections \ref{sez2}, \ref{sez3}, \ref{sez4} and \ref{sez5} we construct a sequence of random walks $\{W_n(t) \}_n$ on the complex plane and prove a representation for the solution of Eq.\ \eqref{e1} in terms of the limit of expectations with respect to the measure associated to $W_n$.
In section \ref{sez6} we give a semigroup formulation of these problems. In section \ref{sez7} we generalized these results to the case where Eq.\ \eqref{e1} is restricted to the half real line $[0,+\infty)$ or to a bounded interval $[0,L]\subset \bR$ and different boundary conditions are considered.

\section{Preliminaries}\label{sez2}

Let $(\Omega, \cF, \bP)$ be a probability space. Let $\alpha$ be a complex number and $N > 2$ a given integer.
\\
Let $R(N)= \{e^{2 i \pi k/N},\ k=0,1,\dots, N-1\}$ be the roots of the unity.
Then we consider the random variable $\xi$ that has uniform distribution on the set $\alpha^{1/N} R(N)$:
\begin{equation}
\label{e2}
\bE[f(\xi)] = \frac{1}{N} \, \sum_{k=0}^{N-1} f(\alpha^{1/N} e^{2 i \pi k/ N}).
\end{equation}

\begin{lemma}\label{lp1}
The random variable $\xi$ has finite moments of every order
\begin{equation}
\label{e3}
\bE[\xi^m] =
\begin{cases}
\alpha^{m/N}, & m = n N,\ n \in \bN,
\\
0, & \text{otherwise}
\end{cases}
\end{equation}
\end{lemma}

\begin{proof}
We compute
\begin{align*}
\bE[\xi^m] = \frac{1}{N} \sum_{k=0}^{N-1} \alpha^{m/N} e^{2 i \pi m k/N};
\end{align*}
if $m/N = n \in \bN$, since $e^{2 i \pi n} = 1$ then each term in the sum is equal to 1;
in the other case, we employ the trigonometric sum
\begin{align*}
\sum_{k=0}^{N-1}  e^{2 i \pi m k/N} = \frac{1 - e^{2 i \pi m}}{1 - e^{2 i \pi m/N}} = 0.
\end{align*}
\end{proof}

In particular, its characteristic function is
\begin{align*}
\psi_\xi(\lambda) = \frac{1}{N} \, \sum_{k=0}^{N-1} \exp(i\alpha^{1/N} e^{2 i \pi k/ N}).
\end{align*}
Further, we may compute the absolute moments of $\xi$; then we get:
\begin{align*}
\bE[|\xi|^m] = |\alpha|^{m/N}.
\end{align*}

We can identify $\xi$ with a random vector in the real plane, so that $\xi$ is a centered random variable having covariance matrix
\begin{align*}
\frac12 \, |\alpha|^{2/N} \, I
\end{align*}
where $I$ is the $2\times2$ identity matrix.

\medskip

In the spirit of Hochberg \cite{Hochberg1978} and Burdzy \cite{Burdzy1995} we consider a nonstandard central limit theorem for a sequence of i.i.d. copies of the random variable $\xi$.
We explicitly note that the scaling exponent $1/N$ is weaker than that of the classical CLT and is related to the order of spatial derivative in \eqref{e1}.
We also note that, for $N > 2$, the function $\exp(c x^N)$ is not a well defined characteristic function.
However, we shall see that, in some weak sense, it is associated to a scaling limit for the random walk generated by $\xi$,
hence we shall use for it the name of {\em stable distribution} in analogy with the case $N \le 2$.

\begin{theorem}
\label{t1}
Let $\{\xi_j,\ j \in \bN\}$ be a sequence of i.i.d.\ random variable having uniform distribution on the set $\alpha^{1/N} R(N)$ as in \eqref{e2}.
Let $S_n$ be the random walk defined by the $\{\xi_j\}$, i.e.,
\begin{align*}
S_n = \sum_{j=1}^n \xi_j.
\end{align*}
Then the distribution of the normalized random walk
\begin{align*}
\tilde S_n = \frac{1}{n^{1/N}} S_n
\end{align*}
converges to a stable distribution of order $N$ in the sense that
\begin{equation}
\lim_{n \to \infty} \bE[\exp(i \lambda \tilde S_n)] = \exp \left(\frac{i^N \alpha}{N!} \lambda^N \right).
\end{equation}
\end{theorem}

\begin{proof}
We get
\begin{align*}
\bE[\exp(i \lambda \tilde S_n)]  =
\bE\left[ \exp\left(i\lambda \frac{1}{n^{1/N}}\sum_{j=1}^n\xi_{j}\right) \right]
= \prod_{j=1}^n \bE\left[ \exp\left(\frac{i\lambda \xi_{j} }{n^{1/N}} \right) \right]
= \left( \bE\left[ \exp\left( \frac{i\lambda \xi}{n^{1/N}} \right) \right] \right)^n
\end{align*}
Now, one has that 
\begin{align*}
\bE \left[ e^{ \frac{i\lambda \xi}{n^{1/N}} } \right] = \frac{1}{N} \sum _{k=0}^{N-1} \exp \left({ i\frac{\lambda}{n^{1/N} } \alpha^{1/N}e^{ik\frac{2\pi}{N}}}\right)
\end{align*}
and
\begin{align*}
\left( \bE \left[ e^{ \frac{i\lambda \xi}{n^{1/N}} } \right] \right)^n = \exp \left({n\log\left(  \bE \left[ e^{  \frac{i\lambda \xi}{n^{1/N}}  }  \right] \right)} \right).
\end{align*}
For $n\to \infty$ (in the sequel, we write $f(n) \sim g(n)$ if the two sequences have the same behavior at infinity, i.e., $\displaystyle \lim_{n \to \infty} \frac{f(n)}{g(n)} = 1$)
\begin{multline*}
\log\left(  \bE[ e^{  \frac{i\lambda \xi}{n^{1/N}}  }  ]\right) =\log\left(1+\frac{1}{N}\sum _{k=0}^{N-1} e^{ i\frac{\lambda}{n^{1/N} } \alpha^{1/N}e^{ik\frac{2\pi}{N}}   } -1\right)  
  \sim  \frac{1}{N}\sum _{k=0}^{N-1} \left(  e^{ i\frac{\lambda}{n^{1/N} } \alpha^{1/N}e^{ik\frac{2\pi}{N}}   } -1  \right) \\
 \sim  \frac{1}{N}\sum _{k=0}^{N-1} \frac{1}{N!}\left( i\frac{\lambda}{n^{1/N}}\alpha^{1/N} e^{ik\frac{2\pi}{N}} \right)^N 
  =  \frac{1}{N}\sum _{k=0}^{N-1} \frac{1}{N!}\frac{(i)^N\lambda^N\alpha}{n}=\frac{1}{N!}\frac{(i)^N\lambda^N\alpha}{n}
\end{multline*}
As a consequence we get the thesis: 
\begin{align*}
\lim_{n\to\infty} \bE \left[ e^{i\lambda \tilde S_{ n}} \right] = \lim_{n\to\infty} \exp \left( {n\log\left(  \bE \left[ e^{  \frac{i\lambda \xi}{n^{1/N}}  } \right] \right)} \right) = \exp\left(\frac{(i)^N\lambda^N\alpha}{N!}\right).
\end{align*}
\end{proof}

Notice that in the case $N=2$ the statement is completely equivalent to the classical CLT. In case $N=4$ a similar result was proved in Burdzy \cite[Theorem 3.1 (iii)]{Burdzy1995}.
A central limit theorem for the sum of i.i.d.\ random variables under a {\em signed} measure is proved in Hochberg \cite[Section 5]{Hochberg1978}.

\medskip

We proceed with the analysis of the moments of the normalized random walk $\tilde S_n$.
We are in particular interested to the asymptotic behavior of such moments and, according to \eqref{e3}, we expect that the moment of order $m$ vanishes if $m$ is not a multiple of $N$.
Next result provides an answer to these questions.

\begin{theorem}
Fix $m \in \bN$ and assume that $n$ is large ($n > m$).
Then the $m$-moment of $\tilde S_n$ satisfies
\begin{align*}
\bE[(\tilde S_n)^m] =
\begin{cases}
\left(\frac{\alpha}{N!}\right)^{m/N} \frac{m!}{(m/N)!} + R_n,  & m = M N,\ M \in \bN, \, 
\\
0, & \text{otherwise}
\end{cases}
\end{align*}
where $\lim_{n\to\infty}R_n=0$.
\end{theorem}

\begin{proof}
We have that 
$$\bE[(\tilde S_n)^k] =(-i)^k\frac{d^k}{d\lambda^k} \psi_n(0),$$
where $\psi_n$ is the characteristic function of $\tilde S_n$,  given by
\begin{align*}
\psi_n(\lambda) = \left( \bE\left[\exp\left(\frac{1}{n^{1/N}} i \lambda \xi\right)\right] \right)^n = \left( \psi_\xi\left(\frac{\lambda}{n^{1/N}}\right) \right)^{n}.
\end{align*}
By Fa\'a di Bruno's formula, for $n>k$
\begin{multline*}
\frac{d^k}{d\lambda^k} \psi_n(\lambda)
\\
=
\sum \frac{k!}{m_1!m_2! \cdots m_k!}\frac{n!}{n-(m_1+m_2+\dots+m_k)}\left(\psi_\xi(\lambda/n^{1/N})\right)^{n-(m_1+m_2+\dots+m_k)}\Pi_{j=1}^k\left(\frac{\psi_\xi^{(j)}(\lambda/n^{1/N})}{j!n^{j/N}}\right)^{m_j}
\end{multline*}
where the sum is over the $k-$ple of non-negative integers $(m_1,m_2,...,m_k)$ such that $m_1+2m_2+\dots +km_k=k$.
In particular we have:
\begin{equation}\label{FaadiBruno}
\frac{d^k}{d\lambda^k} \psi_n(0)=\sum \frac{k!}{m_1!m_2! \cdots m_k!}\frac{n!}{n-(m_1+m_2+\dots+m_k)}\Pi_{j=1}^k\left(\frac{\psi_\xi^{(j)}(0)}{j!n^{j/N}}\right)^{m_j}.
\end{equation}
Since $\psi_\xi^{(j)}(0)=(i)^j\bE[\xi^j]$, and $\bE[\xi^j]\neq 0$ iff $j=mN$, with $m\in \bN$, then the product $\Pi_{j=1}^k\left(\frac{\psi_\xi^{(j)}(0)}{j!n^{j/N}}\right)^{m_j}$ is non vanishing iff $m_j=0$ for $j\neq mN$ and $k=Nm_N+2Nm_{2N}+...$, i.e. if $k$ has is a multiple of $N$.\\
On the other hand, for $k=N$ the only term in the sum which does not vanish is the one corresponding to $m_j=0$ for $j\neq N$ and $m_N=1$, and we have
$$\frac{d^N}{d\lambda^N} \psi_n(0)=\frac{\alpha}{(-i)^N}.$$
For $k=2N$, two terms do not vanish:  the first for $m_j=0$ if $j\neq N$ and $m_N=2$, the second for $m_j=0$ if $j\neq 2N$ and $m_{2N}=1$, giving:
$$\frac{d^{2N}}{d\lambda^{2N}} \psi_n(0)=\frac{2N!}{2!(N!)^2}\frac{n(n-1)}{n^2}\frac{\alpha^2}{(-i)^{2N}}+\frac{1}{n}\frac{\alpha^2}{(-i)^{2N}}.$$
More generally, for $k=MN$, we have :
$$\frac{d^{MN}}{d\lambda^{MN}} \psi_n(0)=\frac{(MN)!}{M!}\frac{n!}{n-M}\frac{\alpha^M}{(-i)^{MN}N^Mn^M}+R_n$$
where $R_n\to 0 $ as $n\to \infty$.

\end{proof}

\medskip

Next result makes precise the analysis  of the asymptotic behavior of the limit in Theorem \ref{t1}.
Since its proof is rather technical, we postpone the proof to the Appendix, where it is developed in parallel with similar computations we shall give later in Section \ref{sez4}.

\begin{theorem}\label{theo4}
The following asymptotic limit holds:
\begin{equation}
\label{e4}
\lim_{n \to \infty} n \, \left[\psi_n(\lambda) - \exp(\frac{i^N \alpha}{N!} \lambda^N)\right] =  (-1)^N\left(\frac{1}{(2N)!}-\frac{1}{2(N!)^2}\right)  \alpha^2 \lambda^{2N} \, \exp(\frac{i^N \alpha}{N!} \lambda^N).
\end{equation}
\end{theorem}

\section{The random walk on the complex plane}\label{sez3}

In this section, we consider the random walk $S_n$ defined above. With no
claim for completeness, we discuss some interesting aspects of the motion.
\\
First, let us consider the case $N=3$. The walk $S_n$ occurs on the
regular lattice generated by the vectors $\{(1,0),\
(-\frac12,\frac{\sqrt{3}}{2}),\ (-\frac12,-\frac{\sqrt{3}}{2})\}$,
considered as a {\em directed} graph.
Therefore, the motion is $3$-periodic, and a return to the origin only
happens if the same number of steps is made in every direction.
Therefore, we compute
\begin{align*}
\bP(S_{3m}=0) = \frac{(3m)!}{(m!)^3} \frac{1}{3^{3m}}
\end{align*}
and Stirling's formula implies
$\bP(S_{3m}=0) \sim \frac{1}{2 \, \pi \, m}$;
hence the expected number of returns to the origin is
\begin{align*}
\sum_{m=1}^\infty \bP(S_{3m}=0) \sim \sum_{m=1}^\infty \frac1m = +\infty
\end{align*}
and the process is recurrent.
\\
The case $N=4$ corresponds to the standard, two-dimensional random walk.
We know that the motion is $2$-periodic and it moves on the lattice
$\bZ^2$ (this time considered as an {\em undirected} graph). Finally, the
motion is recurrent.
\\
Unfortunately, it is not true that the motion is recurrent for every $N$.
Let us see what happens for $N=5$. In this case,
the motion is again $5$-periodic and the only way to return to the origin
is taking the same number of steps in each direction.
Hence, again by an application of Stirling's formula,
\begin{align*}
\bP(S_{5m}=0) = \frac{(5m)!}{(m!)^5} \frac{1}{5^{5m}} \sim
\frac{\sqrt{5}}{(2 \, \pi \, m)^2}
\end{align*}
and the expected number of returns is finite:
\begin{align*}
\sum_{m=1}^\infty \bP(S_{5m}=0)  \sim \sum_{m=1}^\infty \frac{\sqrt{5}}{(2
\, \pi \, m)^2} < \infty
\end{align*}
so the process is transient.
\\
However, we can prove that for every $N\ge3$ the random walk $S_n$ is
neighborhood-recurrent, i.e., for any radii $\varepsilon$, the motion
returns infinitely often in the ball of radii $r$ centered in the origin.
Since this is a weaker result than being recurrent, the following result
is interesting only for $N \ge 5$.

\begin{proposition}
Let $N\ge 5$. The process $\{S_n\}$ is neighborhood-recurrent, i.e., for
every $x$ in the lattice generated by the basis $\{\alpha^{1/N} e^{2 \pi i
k/N},\ k=0,1,\dots,N-1\}$ it holds
\begin{align*}
\bP(|S_n - x| \le \varepsilon \quad \text{infinitely often}) = 1.
\end{align*}
\end{proposition}

\begin{proof}
We first remark that the classical CLT apply to the sequence $\{\xi_k\}$
of random vectors in the plane, so that for large $n$ it holds
\begin{align*}
\frac{1}{\sqrt{n}} \, S_n \sim \cN\left(0, \frac12 \, |\alpha|^{2/N} \, I\right).
\end{align*}
\\
We compute
\begin{align*}
\bP(|S_n| \le \varepsilon) &\simeq \int_{B(0,\varepsilon\, n^{-1/2}) }
\cN\left(0, \frac12 \, |\alpha|^{2/N} \, I\right)
({\rm d}x)
\\
&= \int_{B(0,\varepsilon\, n^{-1/2})} \exp\left(-\frac{1}{|\alpha|^{2/N}}
|x|^2\right) \frac{{\rm d}x}{\pi \, |\alpha|^{2/N}}
\\
&= \int_0^{\varepsilon\, n^{-1/2}} 2 \pi \, \exp\left(-\frac{1}{
|\alpha|^{2/N}} \rho^2\right) \frac{\rho \,  {\rm d}\rho}{\pi \,
|\alpha|^{2/N}}
= |\alpha|^{2/N} \left( 1 -
\exp\left(-\frac{\varepsilon^2}{|\alpha|^{2/N}} \frac1n \right) \right)
\end{align*}
The right-hand side converges to 0, for large $n$, as follows
\begin{align*}
\lim_{n \to \infty} n \, |\alpha|^{2/N} \,\left[1 -
\exp\left(-\frac{\varepsilon^2}{|\alpha|^{2/N}} \frac1n \right)\right] =
\varepsilon^2;
\end{align*}
it follows that the series
\begin{align}\label{e:diverge}
\sum_{n=1}^{\infty} \bP(|S_n| \le \varepsilon) \sim \sum_{n=1}^\infty
\left[1 - \exp\left(-\frac{\varepsilon^2}{|\alpha|^{2/N}} \frac1n
\right)\right] = +\infty
\end{align}
diverges for all $\varepsilon > 0$.
\\
The last part of the proof is inspired by Chung \cite[Theorem 8.3.2]{Chung2001}.
Set for $z \in \bR^2$ and $\varepsilon > 0$
\begin{align*}
p_{\epsilon,n}(z) = \bP(|S_n - z| > \varepsilon \text{\ for all $n \ge m$});
\end{align*}
set $F$ the event $(|S_n - x| \le \varepsilon \quad \text{infinitely often})$
then 
\begin{align}
\label{e:diverge2}
F^c = \bigcup_{m \in \bN} (|S_n - z| > \varepsilon \text{\ for all $n \ge m$}) \quad \text{hence}\quad \bP(F^c) = \lim_{m \to \infty} p_{\varepsilon,m}(0).
\end{align}
Let $A_{m,k} = (|S_m| \le \varepsilon,\ |S_n| > \varepsilon \text{\ for all $n \ge m+k$})$; notice that
$A_{m,k}$ and $A_{m',k}$ are disjoint provided that $m' > m+k$, hence 
each sequence of events $\{A_{m,m+jk}\}_{j=0}^\infty$, for $m=1,\dots,k$, is made of disjoint events and 
\begin{align*}
\sum_{n=1}^\infty \bP(A_{n,k}) = \sum_{m=1}^k \sum_{j=0}^\infty \bP(A_{m,m+jk}) = \sum_{m=1}^k \bP\left( \bigcup_{j=0}^\infty A_{m,m+jk} \right) \le k.
\end{align*}
On the other hand, it holds
\begin{align*}
\sum_{n=1}^\infty \bP (|S_m| \le \varepsilon,\ |S_n| > \varepsilon &\text{\ for all $n \ge m+k$}) = \sum_{n=1}^\infty \bP (|S_m| \le \varepsilon,\ |S_n-S_m| > 2\varepsilon \text{\ for all $n \ge m+k$}) 
\\
&= \sum_{n=1}^\infty \bP (|S_m| \le \varepsilon) \, \bP( |S_n-S_m| > 2\varepsilon \text{\ for all $n \ge m+k$})
\\
&= \sum_{n=1}^\infty \bP (|S_m| \le \varepsilon) \, \bP( |S_{n-m}| > 2\varepsilon \text{\ for all $n-m \ge k$})
= \sum_{n=1}^\infty \bP (|S_m| \le \varepsilon) \, p_{2\varepsilon,k}(0)
\end{align*}
where we use that $S_n - S_m$ is independent of $S_n$ and it has the same distribution as $S_{n-m}$.
Therefore, we have proved that
\begin{align*}
\forall\,k \ge 1, \qquad \sum_{n=1}^\infty \bP (|S_m| \le \varepsilon) \, p_{2\varepsilon,k}(0) \le k
\end{align*}
and \eqref{e:diverge} implies that, for all $k\ge 1$, $p_{2\varepsilon,k}(0) = 0$. Recalling \eqref{e:diverge2} we get the thesis.
\end{proof}

\begin{remark}
With a similar computation, we see that for large $n$
\begin{equation}
\label{e7}
\bP(|\tilde S_n| > \varepsilon) \sim
\exp\left(-\frac{\varepsilon^2}{|\alpha|^{2/N}} n^{\frac2N-1} \right)
\end{equation}
which converges to 1 since $N > 2$.
\end{remark}

\noindent
\begin{minipage}{.45\textwidth}
\begin{center}
\includegraphics[width=\textwidth]{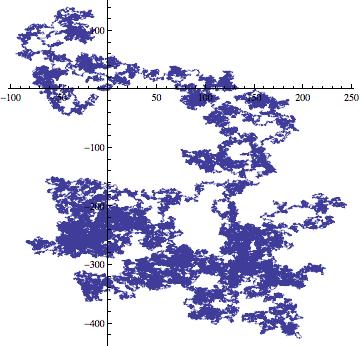}
\end{center}
{\footnotesize A simulation for the path of $S_n$ in case $N=5$.}
\end{minipage}
\hskip.08\textwidth\begin{minipage}{.45\textwidth}
\begin{center}
\includegraphics[width=\textwidth]{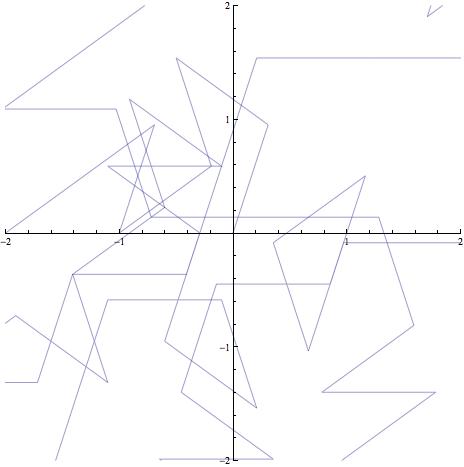}
\end{center}
{\footnotesize The same simulation; an inspection of a neighborhood of the
origin.}
\end{minipage}
\par\vskip 1\baselineskip\noindent

We provide a description of the symmetry properties of the random walk $S_n$, inherited by those of the set $R(N)=\{e^{2 i \pi k/N},\ k=0,1,\dots, N-1\}$.
\begin{proposition}
Let $U:\bC\to \bC$ be a linear map such that $U(R(N))=R(N)$. Then $U$ is a symmetry  transformation for $S_n$, i.e 
$$
\bP(S_n=z)=\bP(S_n=U(z)),\qquad z\in \bC.
$$
\end{proposition}

\begin{proof}
The result can be proved by induction on $n$:
for $n=1$ the result follows by the assumption $U(R(N))=R(N)$;
for general $n$, by using the inductive hypothesis, one can write:
\begin{eqnarray}
\bP(S_n=z)&=&\sum_{j=1}^N\bP(S_{n-1}=z-e^{i\frac{2\pi}{N}j})=
\sum_{j=1}^N\bP(S_{n-1}=U(z)-U(e^{i\frac{2\pi}{N}j}))\nonumber\\
&=&\sum_{j=1}^N\bP(S_{n-1}=U(z)-e^{i\frac{2\pi}{N}j})=\bP(S_n=U(z)).\nonumber
\end{eqnarray}
\end{proof}

\section{A family of complex jump processes}\label{sez4}

In case $N=2$, the limit of the random walk $\tilde S_n$ is a Wiener process; to be precise, since in the definition of $\tilde S_n$ no time is involved, it converges to the Wiener process at time $t=1$.
It is possible to extend this result to general times; however, it is not possible to talk about the limit process in case $N > 2$.
In this section, we construct a family of random walks $W_n(t)$ that generalizes, in a suitable sense, $\tilde S_n$ to a continuous time process, taking into account the limit expressed in Theorem \ref{t1}.

Let us start by considering the time interval $[0,1]$. Let $\{\xi_j\}$ be a sequence of independent copies of the random variable $\xi$ defined in \eqref{e2}.
Then for any $n \in \bN$ we set
\begin{align}
X_n(0) &= 0;\nonumber \\
X_n(\tfrac{k}{n}) &=  \frac{1}{n^{1/N}} \sum_{j=1}^k \xi_j, \qquad k=1,\dots, n,\label{Xn} \\
X_n(t) &= X_n(\tfrac{k}{n}), \qquad  t \in [\tfrac{k}{n}, \tfrac{k+1}{n}).\nonumber
\end{align}
Let us extend now $X_n(t)$ to a process $W_n(t)$ for values of $t\in (-\infty,+\infty)$. For $t > 0$ we set $\lfloor t \rfloor$ the integer part of $t$ and
\begin{align}
W_n(t)&= X^{(1)}_n(t),\qquad &t\in [0,1], \nonumber \\
W_n(t)&= W_n(1)+X^{(2)}_n(t- 1),\qquad &t\in [1,2], \label{Wn}  \\
\intertext{and, in general,}
W_n(t)&= 
W_n(\lfloor t \rfloor)+X^{(\lfloor t \rfloor + 1)}_n(t-\lfloor t \rfloor), &t \ge 0, \nonumber
\end{align}
while for negative times we set (with the convention that $\lfloor -t \rfloor = - \lfloor t \rfloor$)
\begin{align}
W_n(-t)&= e^{i\pi/N} X^{(-1)}_n(t), & t\in [0,1], \nonumber \\
W_n(-t)&= W_n(-1)+e^{i\pi/N}X^{(-2)}_n(t-1), & t\in [1,2], \label{Wnbis}\\
\intertext{and, in general,}
W_n(-t)&= 
W_n(\lfloor -t \rfloor) + e^{i \pi / N} X^{( \lfloor -t \rfloor - 1)}(\lfloor -t \rfloor-(-t)), & t \ge 0, \nonumber 
\end{align}
where $X^{(1)}_n,X^{(2)}_n, \dots ,X^{(-1)}_n,X^{(-2)}_n, \dots$ are i.i.d.\ copies of $X_n$ in \eqref{Xn}.
We remark that $W_n$ can be seen as the extension to continuous time of the random walk $\{\tilde S_n\}$, since we have the following identity for the laws
\begin{align}\label{eq:W=S}
W_n(t) \stackrel{\cL}{=} \left( \frac{\lfloor n t \rfloor}{n} \right)^{1/N} \tilde S_{\lfloor n t \rfloor}, \qquad \text{for $t > 0$.}
\end{align}

The sequence of processes $\{W_n\}$ shall converge in a very weak sense 
to a sort of $N$-stable process (which, we note again, does not really exist for $N > 2$).
The result is analog to what is proved in Theorem \ref{t1} for the normalized random walk $\tilde S_n$.

\begin{theorem}
\label{t2}
For any $t \in (-\infty,+\infty)$ and $\lambda \in \bC$,

\begin{equation}
\label{e5}
\lim_{n \to \infty} \bE[\exp(i \lambda W_n(t))] = \exp\left(i^N \frac{\lambda^N}{N!} \alpha t\right).
\end{equation}
\end{theorem}

For clearness of exposition, we postpone the proofs of this and next result to the appendix, since their details are not required for a comprehension of the sequel.
However, let us give a heuristic idea: for large $n$ it holds $\frac{\lfloor n t \rfloor}{n} \simeq t$, therefore we expect that
\begin{align*}
\bE[e^{i \lambda W_n(t)}] = \bE[ e^{i \lambda \left( \frac{\lfloor n t \rfloor}{n} \right)^{1/N} \tilde S_{\lfloor n t \rfloor} } ]  \sim \bE[ e^{i \lambda t^{1/N} \tilde S_{\lfloor n t \rfloor}}] \sim \exp\left(i^N \frac{(\lambda t^{1/N})^N}{N!} \alpha \right)
\end{align*}
that is \eqref{e5}.
In next result, as we did for the random walk $\{S_n\}$, we  study the error term in the convergence proved above.

\begin{lemma}
\label{lemmat2}
For any $t \in  (-\infty,+\infty)$,  one has
\begin{align*}
\bE[\exp(i \lambda W_n(t))] - \exp\left(i^N \frac{\lambda^N}{N!} \alpha t\right) = f_n(t) + g_n(t),
\end{align*}
where for $n\to\infty$
\begin{align*}
f_n(t) &\sim \frac{\alpha^2 |t|}{n} (i \, \lambda)^{2N}  \exp\left(i^N \frac{\lambda^N}{N!} \alpha t\right)\left(\frac{1}{(2N)!}-\frac{1}{2(N!)^2}\right)
\end{align*}
and
\begin{align*}
|g_n(t)| &\le \frac{1}{n}\frac{|\alpha \lambda^N|}{N!} \left| \exp\left(i^N \frac{\lambda^N}{N!} \alpha t\right) \right|.
\end{align*}
\end{lemma}

As stated above, we postpone the proof of this result to the appendix; here we record the following  direct consequence of the lemma.
\begin{corollary}\label{cor1}
For any $\epsilon>0$ there exists a $n_\epsilon\in\bN$ such that for $n> n_\epsilon$:
\begin{align*}
\left|\bE[\exp(i \lambda W_n(t))] - \exp\left(i^N \frac{\lambda^N}{N!} \alpha t\right)\right| <(1+\epsilon) \frac{K(t,\alpha,\lambda)}{n},
\end{align*}
where
\begin{align}\label{K}
K(t,\alpha,\lambda)= \left| \exp\left(i^N \frac{\lambda^N}{N!} \alpha t\right) \right| \left(\frac{|\alpha \lambda^N|}{N!}+|\alpha^2 t\lambda^{2N}|\left(\frac{1}{2(N!)^2}-\frac{1}{(2N)!}\right)\right).
\end{align}
\end{corollary}

\section{Solution of higher order PDEs}\label{sez5}

In this section we show that the weak limit in distribution proved in Theorem \ref{t2} is equivalent to a representation formula for the solution of a $N$-order equation with possibly complex coefficients.
\\
Specifically, we are concerned with the following complex valued parabolic PDE
\begin{equation}
\begin{aligned}
&\frac{\partial}{\partial t}u(t,x) = \frac{\alpha}{N!} \frac{\partial^N}{\partial x^N}u(t,x), 
\\
&u(t_0,x) = f(x), \qquad \phantom{\frac{\partial^N}{\partial x^N}}&x \in \bR.
\end{aligned}
\label{e6}
\end{equation}
We are going to show that for a suitable class of initial datum $f$, the limit
\begin{equation}
\label{e8}
u(t,x) = \lim_{n \to \infty} \bE[f(x + W_n(t-t_0))]
\end{equation}
is well defined for any $x \in \bR$ and $t$ in a suitable neighborhood of $t_0$, and it provides a representation for the solution of \eqref{e6}.

\begin{hypothesis}
\label{h1}
The class of admissible initial datum is defined by the (complex-valued) function $f(x)$, defined for $x \in \bR$, of the form
\begin{align*}
f(x) = \int_{\bR} e^{i x y} \, {\rm d}\mu(y),
\end{align*}
where $\mu$ is a measure of bounded variation on $\bR$ satisfying the following assumptions:
\begin{enumerate}
\item[1.] $\displaystyle \int_{\bR} |e^{i x z}| \, {\rm d}|\mu|(x) < \infty$ for all $z \in \bC$,
\item[2.] there exists a time interval $(T_1,T_2)$, with $T_1 < t_0 < T_2 \in \bR$, such that
\begin{align*}
 \int_{\bR} \left|\exp\left(i^N \alpha  \frac{x^N}{N!}(t-t_0))\right)\right| \, {\rm d}|\mu|(x) < \infty
 \end{align*}
  for all $t \in (T_1,T_2)$.
\end{enumerate}
In the conditions above, $|\mu|$ stands for the total variation of $\mu$. In order to emphasize the dependence on time, we shall write the class of initial data by $D(T_1,T_2)$.
\\
We shall also require, sometimes, the following additional assumption 
\begin{enumerate}
\item[3.] $\displaystyle \int_{\bR} \left|x^{N} \, \exp\left(i^N \alpha  \frac{x^N}{N!}(t-t_0)\right) \right| \, {\rm d}|\mu|(x) < \infty$ for all $t \in (T_1,T_2)$.
\end{enumerate}
\end{hypothesis}

\medskip

\begin{remark}
Condition 1.\ assures that the function $f$ can be extended to the complex plane in the following way:
\begin{align*}
f(z) \equiv \int_{\bR} e^{i z y} \, {\rm d}\mu(y), \qquad z \in \bC,
\end{align*}
since the integral above is absolutely convergent.
Therefore, we may give a meaning to $f(x + W_n(t))$, where $W_n$ denotes  the random process  introduced above.
Notice that the computation in \eqref{e7} implies that $f$ shall be extended to the {\em whole} complex plane in order to ensure that $f(x + W_n(t))$ is well defined.
\end{remark}

\begin{remark}\label{rem2}
Condition 2.\ depends on time and, in particular, on $\alpha$. This is because we do not give any condition on $\alpha$ and, therefore, the solution may explode in finite time.
However, the length of the interval $(T_1,T_2)$ is independent of $t_0$, so that our construction is invariant under time shifts. In particular, with no loss of generality we can choose $t_0 = 0$. 
\\
If the condition is satisfied, then the function $x \mapsto u(t,x)$ is well defined for all $x \in \bR$ and continuous.
\\
We cannot control the supremum norm of $u(t,x)$ with that of $u(0,x) = f(x)$ unless the stochastic process is concentrated on the real line, which happens in the case $N=2$ and $\alpha \in \bR_+$.
\\
Indeed, for $N>2$ the solutions of equation \eqref{e6} do not satisfy a maximum principle.
\end{remark}

We can now state our main result.

\begin{theorem}\label{teo10}
Suppose that $f \in D(T_1,T_2)$ is an admissible initial datum that satisfies conditions 1.\ and 2.\ in Hypothesis \ref{h1} above, for $T_1 < t_0 < T_2$.
The function $u(t,x)$ defined in \eqref{e8} is a representation of the solution of the parabolic problem \eqref{e6} for any time $t \in (T_1,T_2)$ in the sense that 
\begin{equation}\label{solHoc}
u(t,x)= \int_{\bR} e^{i\, x\, y} \exp\left(\frac{i^N \alpha}{N!} (t -t_0)y^N \right) \, {\rm d}\mu(y)
\end{equation}
and the integral in \eqref{solHoc} is absolutely convergent.

Suppose further that condition 3.\ in Hypothesis \ref{h1} is satisfied. Then the solution $u(t,x)$ is a classical solution for the problem \eqref{e6}.
\end{theorem}

\begin{proof}
As stated in Remark \ref{rem2}, we fix $t_0 = 0$ for simplicity.
By Hypothesis \ref{h1}.1.\ the integral
\begin{align*}
f(x+W_n(t)) = \int_{\bR} \exp(i \, (x + W_n(t)) \, y) \, {\rm d}\mu(y)
\end{align*}
is well defined for any $x \in \bR$ and for any value of $W_n(t)$.
\\
We remark that $W_n(t)$, for any $n \in \bN$ and $t \in \bR$, can assume only a finite number of possible values; therefore, we have that
\begin{align*}
\bE[f(x + W_n(t))] = \bE\left[ \int_{\bR} \exp(i \, (x + W_n(t)) \, y) \, {\rm d}\mu(y) \right] =  \int_{\bR} e^{i \, x \, y} \bE\left[\exp(i \, W_n(t) \, y)\right]  \, {\rm d}\mu(y)
\end{align*}
which implies
\begin{align*}
\lim_{n \to \infty} \bE[f(x + W_n(t))] = \lim_{n \to \infty}  \int_{\bR} e^{i \, x \, y} \bE\left[\exp(i \, W_n(t) \, y)\right]  \, {\rm d}\mu(y)
\end{align*}
By the dominated convergence theorem, which holds thanks to Theorem \ref{t1} and Hypothesis \ref{h1}.2., we have
\begin{multline*}
\lim_{n \to \infty}  \int_{\bR} e^{i \, x \, y} \bE\left[\exp(i \, W_n(t) \, y)\right]  \, {\rm d}\mu(y)
= \int_{\bR} e^{i \, x \, y} \lim_{n \to \infty} \bE\left[\exp(i \, W_n(t) \, y)\right]  \, {\rm d}\mu(y)
\\
= \int_{\bR} e^{i\, x\, y} \exp\left(\frac{i^N \alpha}{N!} t \, y^N \right) \, {\rm d}\mu(y)
\end{multline*}
and the right-hand side defines the solution \eqref{solHoc} of problem \eqref{e6}.

If we assume further that Hypothesis \ref{h1}.3.\ holds, then we can take the derivative in the formula \eqref{e8} to get
\begin{align*}
\frac{\partial}{\partial t} u(t,x) = \frac{\alpha}{N!} \frac{\partial^N}{\partial x^N} u(t,x) = \frac{i^N \alpha}{N!} \int_{\bR} \lambda^N e^{i x \lambda} \exp\left(\frac{i^N \alpha}{N!} t \lambda^N \right) \, {\rm d}\mu(\lambda)
\end{align*}
and Hypothesis \ref{h1}.3 implies that the integral on the right hand side is absolutely convergent and all the quantities are finite.
\end{proof}

\begin{remark}
Corollary \ref{cor1} allows one to estimate the speed of convergence of the approximated solution $u_n(t,x)=\bE[f(x+W_n(t))]$ to the solution $u(t,x)$. Indeed let us assume that $f\in D(T_1,T_2)$ is an admissible initial datum satisfying conditions 1.\ and 2.\ and 3. in Hypothesis \ref{h1} above, for $T_1 < t_0 < T_2$. Let us assume moreover that 
\begin{equation}\label{hyp-4}
\int_{\bR} \left|x^{2N} \, \exp\left(i^N \alpha  \frac{x^N}{N!}(t-t_0)\right) \right| \, {\rm d}|\mu|(x) < \infty, \qquad\forall t \in (T_1,T_2).
\end{equation}
Then by corollary \ref{cor1} one can see that for any $\epsilon >0$ there exists a $n_\epsilon\in \bN$ such that for $n>n_\epsilon$:
\begin{equation*}
|u(t,x)-u_n(t,x)|\leq (1+\epsilon)\frac{C(t)}{n}\qquad \forall x\in \bR, \, t\in (T_1,T_2)
\end{equation*}
where
\begin{multline}\nonumber
C(t)=\frac{|\alpha|}{N!}\int |x|^N\left|\exp\left(i^N \alpha  \frac{x^N}{N!}(t-t_0)\right) \right| \, {\rm d}|\mu|(x) +\\ +|\alpha|^2(t-t_0)\left(\frac{1}{2(N!)^2}-\frac{1}{(2N)!}\right)\int |x|^{2N}\left|\exp\left(i^N \alpha  \frac{x^N}{N!}(t-t_0)\right) \right| \, {\rm d}|\mu|(x).
\end{multline}
\end{remark}

\section{A semigroup formulation for the $N$-th order Laplacian}\label{sez6}

Let $\cF_0$ be the set of functions $f:\bR\to \bC$ of the form
$$f(x)=\int e^{iyx}d\mu (y),\qquad x\in \bR$$
for every  $\mu$ complex bounded variation measure on $\bR$.
Since the space of bounded complex measures $\cM(\bR)$ is a Banach algebra under convolution in the total variation norm $\|\mu\|$, one has that $\cF_0$ is a Banach algebra under multiplication in the norm $\|f\|_0:=\|\mu\|$, with $f(x)=\int e^{iyx}d\mu (y)$. The elements of $\cF_0$ are bounded continuous functions and we have $\|f\|_\infty\leq \|f\|_0$.

For any $n \in \bN$ let us  
 denote by $\cF_n\subset \cF_0$ the set of functions $f\in \cF_0$ such that $\int_\bR e^{n |x|}d |\mu |(x)<\infty$ endowed with the norm 
$\| f\|_n=\int_\bR e^{n |x|}d |\mu |(x)$.  
One has that,  for $n\leq m$, $\cF_m\subseteq \cF_n$ since
$\|f\|_n\leq \|f\|_m$ for all $f\in \cF_m\cap \cF_n$.

Let $\cD$ be the topological vector space given by $ \cD:=\cap_n \cF_n$,  endowed with the topology defined by the system of neighborhoods of 0  of the form $U_{r,\epsilon}:=\{ f\in \cD\, |\,  \|f\|_i<\epsilon, \forall i\leq r\}$, with $r\in \bN$ and $\epsilon \in \bR$, $\epsilon >0$.
One can see that the topology of $\cD$ can be induced by a metric, for instance by means of the following distance
\begin{equation}\label{metricD}
d(f,g):=\sum_{n=0}^\infty 2^{-n}\frac{\|f-g\|_n}{1+\|f-g\|_n},\qquad f,g\in D,
\end{equation}
and $\cD$ is complete with respect to this metric. 
Further a sequence $\{f_n\}_n$ is of Cauchy type with respect to the metric \eqref{metricD} and converges to the element $f\in \cD$ if and only if it is of Cauchy type and convergent to $f\in \cD$ with respect to $\| \,\cdot\, \|_n$ for any $n\in\bN$ 
\cite{Kolmogorov1957}.

Let us consider now the parabolic problem \eqref{e6} in the case where $t_0=0$ and $\alpha\in \bC$ and $N\in \bN$  are such that 
\begin{equation}\label{cond1}
|e^{\frac{i^N}{N!}\alpha x^Nt}|\leq 1
\end{equation} 
for all $x\in\bR$ and $t\geq 0$. If $N $ is even, this condition is satisfied if and only if $Re (-1)^{N/2}\alpha <0$, 
while if $N$ is odd \eqref{cond1} is satisfied if and only if $\alpha\in \bR$, moreover in this case the inequality \eqref{cond1} holds for any $t\in\bR$. 

Under these assumptions on the parameters of the PDE \eqref{e6}, one has that $\cD\subset D(0,+\infty)$. Indeed given an element $f\in \cD$, one can easily verify that $f$ satisfies conditions 1., 2., and 3.\ of hypothesis \ref{h1}. It is then possible to define a strongly continuous semigroup $T(t):\cD\to \cD$ in terms of equation \eqref{e8}, namely:
\begin{equation}
T(t)f(x):=\lim_{n\to \infty}\bE[f(x+W_n(t))], \qquad f \in \cD.
\end{equation}

The bounded operator $T(t)$ maps a function $f\in \cD$, with $f=\hat \mu$, to a function $f_t\in \cD$, with $f_t=\hat \mu_t$, 
where the measure $\mu_t$ is absolutely continuous with respect to $\mu$ with Radon-Nikodym derivative equals to $e^{\frac{i^N}{N!}\alpha x^nt}$, i.e.:
$$d\mu_t(x)=e^{\frac{i^N}{N!}\alpha x^Nt} d\mu(x).$$
One can easily verify the semigroup law. 
By condition \eqref{cond1} one has that $\|T(t)f\|_n\leq \|f\|_n$ for all $n\in \bN$. Moreover, by dominated convergence theorem, for any $n\in \bN$ one has 
$$\lim_{t\downarrow 0}\| T(t)f-f\|_n=0$$
and the semigroup $T(t)$ 
is strongly continuous in the topology of $\cD$. The generator $A$ of $T(t)$ is a bounded operator on $\cD$ and it is given by 
\begin{equation}\label{generator}
Af(x)=\int e^{iyx}\frac{i^N}{N!}\alpha y^N d\mu (y),\qquad \forall f\in \cD, \, f=\hat \mu.
\end{equation}
 Indeed for any $n\in\bN$, by the dominated convergence theorem, one has:
$$\lim_{t\downarrow 0}\left\|\frac{(T(t)f-f)}{t}-Af   \right\|_n=\lim_{t\downarrow 0}\int\left|\frac{e^{\frac{i^N}{N!}\alpha x^Nt}-1}{t}-\frac{i^N}{N!}\alpha x^N\right|e^{n |x|}d |\mu |(x)=0.$$
Finally, as the generator $A:\cD\to \cD$ is  bounded, then the semigroup $T(t)$ is uniformly continuous. 

\section{The boundary value problem}\label{sez7}
In this section we consider several different boundary value problems associated to equation \eqref{e6}.

Let us consider first of all equation \eqref{e6} on the half real line $\bR^+$ and restrict ourselves to the case where $N$ is even. 
Let us denote by $D$ (resp. $N$) the boundary value problem with Dirichlet (resp. Neumann) boundary conditions.
Given a function $f:\bR^+\to \bC$, it can be extended to an odd function $f_O:\bR\to \bC$ on the real line in the following way
$$ f_O(x)=\begin{cases}
f(x),  & x\geq 0, \\
-f(-x), & x<0.
\end{cases}$$
In a similar way, a function $f:\bR^+\to \bC$ can be extended to an even function $f_E:\bR\to \bC$ as: 
$$ f_E(x)=\begin{cases}
f(x),  & x\geq 0, \\
f(-x), & x<0.
\end{cases}$$
Let us denote by $\cD_D$, resp. $\cD_N$, the following subsets of $\cD$:
\begin{align*}
\cD_D:=& \{f\in \cD\, |\, f(x)=-f(-x)\}
\\
\cD_N:=& \{f\in \cD\, |\, f(x)=f(-x)\}.
\end{align*}
By the regularity of the elements of $\cD$, one can see that the functions $f\in \cD_D$ satisfy the equality $f(0)=0$, while the functions $f\in \cD_N$ satisfy the equality $f'(0)=0$.
Let us consider the restriction $A_D:\cD_D\to \cD_D$ (resp. $A_N:\cD_N\to \cD_N$) of the bounded operator $A\in \cL(\cD)$ given by \eqref{generator} to the subspace $\cD_D\subset \cD$ (resp. $\cD_N\subset \cD$). 
 The following proposition extends the results stated in Theorem \ref{teo10} and gives a representation of the form \eqref{e8} for the solution of the boundary value problems.

\begin{theorem}\label{th-Dir-Neu}
Let $N$ be an even integer. 
The operator $A_{BC}:\cD_{BC}\to \cD_{BC}$, where $BC=D$ or $BC=N$,  is a bounded operator on $\cD_{BC}$ and generates a uniformly continuous semigroup 
$T_{BC}( t):\cD_{BC}\to \cD_{BC}$, given by
$$
T_{BC}(t)f(x)=\lim_{n\to\infty}\bE[f(x+W_n(t))].
$$
\end{theorem}

\begin{proof}
We give the proof in the case of Dirichlet boundary condition. The proof in the case of Neumann boundary condition is completely analogous.\\
Given a function $f\in\cD_D$,  it can be represented as
$f(x)=\int e^{iyx}d\mu (y)$, where $\int e^{n|x|}d|\mu|(x)<\infty$ for all $n\in \bN$. By using the symmetry of $f$, one has
$$
f(x)=\frac{f(x)-f_-(x)}{2}=\frac{1}{2}\int (e^{iyx}-e^{-iyx})d\mu(y)=\frac{1}{2}\int e^{iyx}d\mu(y)-\frac{1}{2}\int e^{iyx}d\mu_-(y),
$$
where $f_-(x):=f(-x)$ and $\mu_-([a,b]) = \mu([-b,-a])$.
\\
The operator $A_D$ is given by:
$$A_Df(x)=Af(x)=\int e^{iyx}\frac{i^N}{N!}\alpha y^Nd\mu (y)$$
on the other hand it is also equal to 
\begin{eqnarray}
A_Df(x)&=&\frac{Af(x)-Af_-(x)}{2}=\frac{1}{2}\int e^{iyx}\frac{i^N}{N!}\alpha y^N d\mu(y)-\frac{1}{2}\int e^{iyx}\frac{i^N}{N!}\alpha y^N d\mu_-(y)\nonumber\\
 &=& \frac{1}{2}\int( e^{iyx}-e^{-iyx})\frac{i^N}{N!}\alpha y^N d\mu(y)=\frac{A_Df(x)-A_Df(-x)}{2}\nonumber
\end{eqnarray}
and one can conclude that $A_Df\in \cD_D$. Moreover $A_D$ is bounded and generates an uniformly continuous semigroup $T_{D}( t):\cD_{D}\to \cD_{D}$, given by
$$
T_{D}(t)f(x)=\lim_{n\to\infty}\bE[f(x+W_n(t))].
$$
\end{proof}

\begin{remark}
This result cannot be extended to the case of $N$ odd. Indeed in this case, the operator $A_D$ (resp. $A_N$ ) does not map $\cD_D$ (resp. $\cD_N$) into itself.
\end{remark}

Let us consider now the parabolic problem \eqref{e6}  on the interval $[0,L]$, with  periodic ($P$) boundary conditions.
\\
Let us denote by $\cD_{P(L)}\subset \cD$ the set of functions $f\in \cD$ that are periodic with period equal to $L$, i.e. 
$$\cD_{P(L)}:=\{f\in \cD\, |\, f(x)=f(x+L),\  
\forall\, x\in \bR\}.$$
Due to the smoothness of $f$, every derivative of $f$ is periodic as well. By the periodicity condition, a function $f\in \cD_{P(L)}$ can be represented as a Fourier series of the form
\begin{equation}\label{periodic}
f(x)=\sum_{k=-\infty}^{+\infty}c_ke^{ik\frac{2\pi}{L}x}, \qquad c_k=\frac{1}{2\pi}\int f(x)e^{-ik\frac{2\pi}{L}x}dx.
\end{equation}
On the other hand $f\in \cD$ has the form $f(x)=\int e^{ixy}d\mu(y)$, and by a comparison with \eqref{periodic},   the measure $\mu$ associated to $f$ has the form 
\begin{align*}
 \mu= \sum_{k=-\infty}^{+\infty}c_k \delta _{2 \pi k/L}, \text{ with } \sum_{k=-\infty}^{+\infty}|c_k|e^{\frac{2\pi}{L}n|k|}<\infty \quad \forall n\in \bN.
\end{align*}

 Let us consider the restriction $A_P:\cD_{P(L)}\to \cD_{P(L)}$ of the bounded operator $A\in \cL(\cD)$ given by \eqref{generator} to the subspace 
$\cD_{P(L)}\subset \cD$. The following holds
\begin{theorem}\label{th-periodic}
For any $N\in \bN$, $N\geq 2$ the operator $A_{P}:\cD_{P(L)}\to \cD_{P(L)}$  is a bounded operator on $\cD_{P(L)}$ and generates a uniformly continuous semigroup 
$T_{P}( t):\cD_{P(L)}\to \cD_{P(L)}$, given by
$$
T_{P}(t)f(x)=\lim_{n\to\infty}\bE[ f(x+W_n(t))], \qquad f\in \cD_{P(L)}.
$$
\end{theorem}

\begin{proof}
It is sufficient to show that $A$ maps $\cD_{P(L)}$ into itself. This can be verified directly by using the definition of the operator $A$ and the particular form of the measure $\mu$ associated to an element $f\in \cD_{P(L)}$:
\begin{align*}
f(x)&=\int e^{ixy}d\mu(y), \qquad \mu=\sum_{k=-\infty}^{\infty}c_k \delta _{2\pi k/L}  \\
Af(x)&=\int e^{ixy}d\nu(y), \qquad d\nu(y)=\frac{i^N}{N!}\alpha y^Nd\mu(y)=\sum_{k=-\infty}^{\infty}\frac{i^N}{N!}\alpha c_k k^N \delta _{2\pi k/L}.
\end{align*}
\end{proof}

We conclude this section with a discussion of the Dirichlet and Neumann boundary conditions on the interval $[0,L]$.
Let us denote by $\cD_D([0,L])$, resp. $\cD_N([0,L])$, the subsets of $\cD_{P(2L)}$ made of functions that are odd, respectively even, on the real line: 
\begin{align*}
\cD_D([0,L])&:= \cD_{P(2L)}\cap \cD_D=\{f\in \cD_{P(2L)}\, |\, f(x)=-f(-x)\} 
\\
\cD_N([0,L])&:=\cD_{P(2L)}\cap \cD_N= \{f\in \cD_{P(2L)}\, |\, f(x)=f(-x)\}.
\end{align*}
One can easily verify that any function $f\in \cD_D([0,L])$, resp. $\cD_N([0,L])$,  satisfies Dirichlet, resp. Neumann, conditions on the boundary of the interval $[0,L]$.
Moreover  the elements of $f\in \cD_D([0,L])$ can be represented in the following form
$$f(x)=\int e^{ixy}d\mu(y), \quad \mu=\sum_{k=-\infty}^{\infty}c_k \delta _{\pi k/L}, \quad c_0=0,\ c_k+c_{-k}=0, \forall k\geq 1,$$
or, equivalently, in the form:
$$f(x)=\sum_{k=1}^\infty i (c_k-c_{-k})\sin (k \frac{\pi}{L} x).$$
Analogously the elements of $f\in \cD_N([0,L])$ can be represented in the following form
$$f(x)=\int e^{ixy}d\mu(y), \quad \mu=\sum_{k=-\infty}^{\infty}c_k \delta _{\pi k/L}, \quad c_k-c_{-k}=0, \forall k\geq 1,$$
or, equivalently, in the form:
$$f(x)=c_0+\sum_{k=1}^\infty i (c_k+c_{-k})\cos (k \frac{\pi}{L}x).$$
Let $A_D([0,L])$, resp. $A_N([0,L])$, be the restriction of the operator $A$ to the subspace $\cD_D([0,L])\subset \cD$, resp. $\cD_N([0,L])\subset \cD$. By considering only even values of the integer $N$, the following holds:

\begin{theorem}\label{th-Dir-Neu-periodic}
Let $N$ be an even integer. The operator $A_{BC}([0,L]):\cD_{BC}([0,L])\to \cD_{BC}([0,L])$, where $BC=D$ or $BC=N$,  is a bounded operator on $\cD_{BC}([0,L])$ and generates a uniformly continuous semigroup 
$T_{BC}( t):\cD_{BC}([0,L])\to \cD_{BC}([0,L])$, given by
$$
T_{BC}(t)f(x)=\lim_{n\to\infty}\bE[ f(x+W_n(t))], \qquad f\in \cD_{BC}([0,L]).
$$
\end{theorem}

The result can be proved by  the same procedure used in the proof of theorems \ref{th-Dir-Neu} and \ref{th-periodic}.

\section*{Acknowledgements}
We are grateful to S. Albeverio, G. Da Prato, L. Tubaro for many interesting discussions.

\section*{Appendix}

For simplicity, fix $t>0$ and consider, first, the characteristic function (recall identity \eqref{eq:W=S} and the computations in the proof of Theorem \ref{t1})
\begin{multline*}
\bE[e^{i\lambda W_{n}(t)}] = \exp \left( \lfloor n t \rfloor \, \log \bE \left[ e^{i \lambda \xi / n^{1/N}} \right] \right)
= \exp \left( \lfloor n t \rfloor \, \log \left( \frac1N \sum_{k=0}^{N-1} \exp\left(\frac{i \lambda e^{2 \pi i k/N}}{n^{1/N}} \right) \right) \right)
\\
= \exp \left( \lfloor n t \rfloor \, \log \left( 1 + \frac1N \sum_{k=0}^{N-1} \left[\exp\left(\frac{i \lambda e^{2 \pi i k/N}}{n^{1/N}} \right) -1 \right]\right) \right)
\end{multline*}
by using the series expansion of $e^x$ and the computations in Lemma \ref{lp1} we get
\begin{align*}
\bE[e^{i\lambda W_{n}(t)}] 
= \exp \left( \lfloor n t \rfloor \, \log \left( 1 + \frac{(i)^N \lambda^N \alpha}{N! \, n} + \frac{(i)^{2N} \lambda^{2N} \alpha^2}{(2N)! \, n^2} + O(\tfrac{1}{n^3}) \right) \right)
\end{align*}
and finally, using the series expansion of $\log(x)$ we obtain
\begin{align}
\label{ea1}
\bE[e^{i\lambda W_{n}(t)}] 
= \exp \left\{ \frac{\lfloor n t \rfloor}{n} \,  \left( \frac{(i)^N \lambda^N \alpha}{N! } \right) + \frac{\lfloor n t \rfloor}{n^2} \,  \left(\frac1{(2N)! } - \frac{1}{2 \,(N!)^2 }  \right) (i)^{2N} \lambda^{2N} \alpha^2 + O(\tfrac{1}{n^2}) \right\}.
\end{align}
For negative times, we have from \eqref{Wnbis} the analog of \eqref{eq:W=S}
\begin{align*}
W_n(-t) \stackrel{\cL}{=} e^{i \pi / N} \left(\frac{\lfloor nt\rfloor}{n} \right)^{1/N} \tilde S_n(\lfloor n t \rfloor);
\end{align*}
hence
\begin{align*}
\bE \left[ e^{i \lambda W_n(-t)} \right] = \bE\left[ e^{i \lambda e^{i \pi / N} W_n(t)} \right]
\end{align*}
and we obtain form \eqref{ea1} the following representation
\begin{align}
\label{ea2}
\bE \left[e^{i\lambda W_{n}(-t)} \right] 
= \exp \left\{ \frac{\lfloor n (-t) \rfloor}{n} \,  \left( \frac{(i)^N \lambda^N \alpha}{N! } \right) + \frac{\lfloor n t \rfloor}{n^2} \,  \left(\frac1{(2N)! } - \frac{1}{2 \,(N!)^2 }  \right) (i)^{2N} \lambda^{2N} \alpha^2 + O(\tfrac{1}{n^2}) \right\}.
\end{align}

\medskip

\subsection*{Proof of Theorem \ref{t2}}
Let $n \to \infty$ in \eqref{ea1} (respectively \eqref{ea2} for negative times) and recall that $\frac{\lfloor n t \rfloor}{n} \to t$ as $n \to \infty$, 
while the other terms in \eqref{ea1} converge to 0.
\qed

\medskip

\subsection*{Proof of Lemma \ref{lemmat2}}
We give the proof for $t > 0$, since the other case follows analogously.
Using again \eqref{ea1} we get
\begin{align*}
\bE[e^{i\lambda W_{n}(t)}] &- e^{\frac{(i)^N \lambda^N \alpha t}{N!}}
=  \exp \left\{ \frac{\lfloor n t \rfloor}{n} \,  \left( \frac{(i)^N \lambda^N \alpha}{N! } \right) + \frac{\lfloor n t \rfloor}{n^2} \,  \left(\frac1{(2N)! } - \frac{1}{2 \,(N!)^2 }  \right) (i)^{2N} \lambda^{2N} \alpha^2 + O(\tfrac{1}{n^2}) \right\} - e^{\frac{(i)^N \lambda^N \alpha t}{N!}}
\\
&= e^{\frac{(i)^N \lambda^N \alpha t}{N!}} \left[\exp \left\{ \frac{\lfloor n t \rfloor}{n} \,  \left( \frac{(i)^N \lambda^N \alpha}{N! } \right) + \frac{\lfloor n t \rfloor}{n^2} \,  \left(\frac1{(2N)! } - \frac{1}{2 \,(N!)^2 }  \right) (i)^{2N} \lambda^{2N} \alpha^2 + O(\tfrac{1}{n^2})   - \frac{(i)^N \lambda^N \alpha t}{N!}\right\}  - 1 \right]
\\
&= e^{\frac{(i)^N \lambda^N \alpha t}{N!}} \left[\exp \left\{ \left(\frac{\lfloor n t \rfloor}{n}-t\right) \,  \left( \frac{(i)^N \lambda^N \alpha}{N! } \right) + \frac{\lfloor n t \rfloor}{n^2} \,  \left(\frac1{(2N)! } - \frac{1}{2 \,(N!)^2 }  \right) (i)^{2N} \lambda^{2N} \alpha^2 + O(\tfrac{1}{n^2})   \right\}  - 1 \right]
\\
&= e^{\frac{(i)^N \lambda^N \alpha t}{N!}} \left[ \left(\frac{\lfloor n t \rfloor}{n}-t\right) \,  \left( \frac{(i)^N \lambda^N \alpha}{N! } \right) + \frac{\lfloor n t \rfloor}{n^2} \,  \left(\frac1{(2N)! } - \frac{1}{2 \,(N!)^2 }  \right) (i)^{2N} \lambda^{2N} \alpha^2 + O(\tfrac{1}{n^2})    \right]
\end{align*}
where in the last passage we use again the series expansion of the exponential function. 
Finally, we obtain the thesis setting
\begin{align*}
f_n(t) = \frac{\lfloor n t \rfloor}{n^2}\left(\frac1{(2N)! } - \frac{1}{2 \,(N!)^2 }  \right) (i)^{2N} \lambda^{2N} \alpha^2 e^{\frac{(i)^N \lambda^N \alpha t}{N!}} + O(\tfrac{1}{n^2})
\end{align*}
and
\begin{align*}
g_n(t) = \left( \frac{\lfloor n t \rfloor}{n} - t \right) \frac{(i)^N \lambda^N \alpha }{N!} e^{\frac{(i)^N \lambda^N \alpha t}{N!}}.
\end{align*}
Notice that $\frac{\lfloor n t \rfloor}{n^2} \sim \frac{t}{n}$ (as stated above, the relation $\sim$ means that the ratio between the functions converges to 1 as $n \to \infty$).
As opposite, the quantity $n \left( \frac{\lfloor n t \rfloor}{n} - t \right)$ does not converge (except in the case $t \in \bZ$); hence, we can only give a bound on this term:
\begin{align*}
|g_n(t)| \le \underbrace{\left| \frac{\lfloor n t \rfloor}{n} - t \right|}_{\le 1} \, \left| \frac{(i)^N \lambda^N \alpha }{N!} e^{\frac{(i)^N \lambda^N \alpha t}{N!}} \right|.
\end{align*}
\qed

\medskip

\subsection*{Proof of Theorem \ref{theo4}}
From \eqref{eq:W=S} for $t=1$, the thesis of Theorem \ref{theo4} follows from \eqref{ea1} with the choice $t=1$:
\begin{align*}
\bE[e^{i\lambda \tilde S_{n}}] 
= \exp \left\{   \left( \frac{(i)^N \lambda^N \alpha}{N! } \right) + \frac1{n} \,  \left(\frac1{(2N)! } - \frac{1}{2 \,(N!)^2 }  \right) (i)^{2N} \lambda^{2N} \alpha^2 + O(\tfrac{1}{n^2}) \right\}.
\end{align*}
Hence, as in the proof of Lemma \ref{lemmat2}, we get
\begin{align*}
\bE[e^{i\lambda \tilde S_{n}}] - e^{ \frac{(i)^N \lambda^N \alpha}{N! } }
&= e^{ \frac{(i)^N \lambda^N \alpha}{N! } } \left[ \exp \left\{  \frac1{n} \,  \left(\frac1{(2N)! } - \frac{1}{2 \,(N!)^2 }  \right) (i)^{2N} \lambda^{2N} \alpha^2 + O(\tfrac{1}{n^2}) \right\} - 1 \right]
\\
&= e^{ \frac{(i)^N \lambda^N \alpha}{N! } } \left[ \frac1{n} \,  \left(\frac1{(2N)! } - \frac{1}{2 \,(N!)^2 }  \right) (i)^{2N} \lambda^{2N} \alpha^2 + O(\tfrac{1}{n^2}) \right]
\end{align*}
and the thesis follows passing to the limit as $n \to \infty$.
\qed



\begin{thebibliography}{10}

\bibitem{Albeverio2008}
S.~A. Albeverio, R.~J. H{\o}egh-Krohn, and S.~Mazzucchi.
\newblock {\em Mathematical theory of {F}eynman path integrals}, volume 523 of
  {\em Lecture Notes in Mathematics}.
\newblock Springer-Verlag, Berlin, second edition, 2008.
\newblock An introduction.

\bibitem{Allouba2002}
H.~Allouba.
\newblock Brownian-time processes: the {PDE} connection. {II}. {A}nd the
  corresponding {F}eynman-{K}ac formula.
\newblock {\em Trans. Amer. Math. Soc.}, 354(11):4627--4637 (electronic), 2002.

\bibitem{Beghin2000}
L.~Beghin, K.~J. Hochberg, and E.~Orsingher.
\newblock Conditional maximal distributions of processes related to
  higher-order heat-type equations.
\newblock {\em Stochastic Process. Appl.}, 85(2):209--223, 2000.

\bibitem{Bochner1955}
S.~Bochner.
\newblock {\em Harmonic analysis and the theory of probability}.
\newblock University of California Press, Berkeley and Los Angeles, 1955.

\bibitem{Burdzy1993}
K.~Burdzy.
\newblock Some path properties of iterated {B}rownian motion.
\newblock In {\em Seminar on {S}tochastic {P}rocesses, 1992 ({S}eattle, {WA},
  1992)}, volume~33 of {\em Progr. Probab.}, pages 67--87. Birkh\"auser Boston,
  Boston, MA, 1993.

\bibitem{Burdzy1995}
K.~Burdzy and A.~M{{\setbox0=\hbox{a}{\ooalign{\hidewidth
  \lower1.5ex\hbox{`}\hidewidth\crcr\unhbox0}}}}drecki.
\newblock An asymptotically {$4$}-stable process.
\newblock In {\em Proceedings of the {C}onference in {H}onor of {J}ean-{P}ierre
  {K}ahane ({O}rsay, 1993)}, number Special Issue, pages 97--117, 1995.

\bibitem{Burdzy1996}
K.~Burdzy and A.~M{{\setbox0=\hbox{a}{\ooalign{\hidewidth
  \lower1.5ex\hbox{`}\hidewidth\crcr\unhbox0}}}}drecki.
\newblock It\^o formula for an asymptotically {$4$}-stable process.
\newblock {\em Ann. Appl. Probab.}, 6(1):200--217, 1996.

\bibitem{Cameron1960/1961}
R.~H. Cameron.
\newblock A family of integrals serving to connect the {W}iener and {F}eynman
  integrals.
\newblock {\em J. Math. and Phys.}, 39:126--140, 1960/1961.

\bibitem{Chung2001}
K.~Chung.
\newblock {\em A Course In Probability Theory}.
\newblock Academic Press, 3ed edition, 2001.

\bibitem{Dynkin2006}
E.~B. Dynkin.
\newblock {\em Theory of {M}arkov processes}.
\newblock Dover Publications Inc., Mineola, NY, 2006.
\newblock Translated from the Russian by D. E. Brown and edited by T.
  K{\"o}v{\'a}ry, Reprint of the 1961 English translation.

\bibitem{Funaki1979}
T.~Funaki.
\newblock Probabilistic construction of the solution of some higher order
  parabolic differential equation.
\newblock {\em Proc. Japan Acad. Ser. A Math. Sci.}, 55(5):176--179, 1979.

\bibitem{Helms1967}
L.~L. Helms.
\newblock Biharmonic functions and {B}rownian motion.
\newblock {\em J. Appl. Probability}, 4:130--136, 1967.

\bibitem{Hida1993}
T.~Hida, H.-H. Kuo, J.~Potthoff, and L.~Streit.
\newblock {\em White noise}, volume 253 of {\em Mathematics and its
  Applications}.
\newblock Kluwer Academic Publishers Group, Dordrecht, 1993.
\newblock An infinite-dimensional calculus.

\bibitem{Hochberg1978}
K.~J. Hochberg.
\newblock A signed measure on path space related to {W}iener measure.
\newblock {\em Ann. Probab.}, 6(3):433--458, 1978.

\bibitem{Hochberg1980}
K.~J. Hochberg.
\newblock Central limit theorem for signed distributions.
\newblock {\em Proc. Amer. Math. Soc.}, 79(2):298--302, 1980.

\bibitem{Hochberg1994}
K.~J. Hochberg and E.~Orsingher.
\newblock The arc-sine law and its analogs for processes governed by signed and
  complex measures.
\newblock {\em Stochastic Process. Appl.}, 52(2):273--292, 1994.

\bibitem{Hochberg1996}
K.~J. Hochberg and E.~Orsingher.
\newblock Composition of stochastic processes governed by higher-order
  parabolic and hyperbolic equations.
\newblock {\em J. Theoret. Probab.}, 9(2):511--532, 1996.

\bibitem{Johnson2000}
G.~W. Johnson and M.~L. Lapidus.
\newblock {\em The {F}eynman integral and {F}eynman's operational calculus}.
\newblock Oxford Mathematical Monographs. The Clarendon Press Oxford University
  Press, New York, 2000.
\newblock Oxford Science Publications.

\bibitem{Karatzas1991}
I.~Karatzas and S.~E. Shreve.
\newblock {\em Brownian motion and stochastic calculus}, volume 113 of {\em
  Graduate Texts in Mathematics}.
\newblock Springer-Verlag, New York, second edition, 1991.

\bibitem{Kolmogorov1957}
A.~N. Kolmogorov and S.~V. Fomin.
\newblock {\em Elements of the theory of functions and functional analysis.
  {V}ol. 1. {M}etric and normed spaces}.
\newblock Graylock Press, Rochester, N. Y., 1957.
\newblock Translated from the first Russian edition by Leo F. Boron.

\bibitem{Krylov1960}
V.~J. Krylov.
\newblock Some properties of the distribution corresponding to the equation
  {$\partial u/\partial t=(-1)^{q+1} \partial ^{2q}u/\partial x^{2q}$}.
\newblock {\em Soviet Math. Dokl.}, 1:760--763, 1960.

\bibitem{Leandre2003}
R.~L{\'e}andre.
\newblock Theory of distribution in the sense of {C}onnes-{H}ida and {F}eynman
  path integral on a manifold.
\newblock {\em Infin. Dimens. Anal. Quantum Probab. Relat. Top.},
  6(4):505--517, 2003.

\bibitem{Leandre2010}
R.~L{\'e}andre.
\newblock Stochastic analysis without probability: study of some basic tools.
\newblock {\em J. Pseudo-Differ. Oper. Appl.}, 1(4):389--400, 2010.

\bibitem{Levin2009}
D.~Levin and T.~Lyons.
\newblock A signed measure on rough paths associated to a {PDE} of high order:
  results and conjectures.
\newblock {\em Rev. Mat. Iberoam.}, 25(3):971--994, 2009.

\bibitem{Leandre2006}
R.~Léandre.
\newblock Path integrals in noncommutative geometry.
\newblock In J.-P. Françoise, G.~L. Naber, and T.~S. Tsun, editors, {\em
  Encyclopedia of Mathematical Physics}, pages 8 -- 12. Academic Press, Oxford,
  2006.

\bibitem{Madrecki1993}
A.~M{\polhk{a}}drecki and M.~Rybaczuk.
\newblock New {F}eynman-{K}ac type formula.
\newblock {\em Rep. Math. Phys.}, 32(3):301--327, 1993.

\bibitem{Mazzucchi2009}
S.~Mazzucchi.
\newblock {\em Mathematical {F}eynman path integrals and their applications}.
\newblock World Scientific Publishing Co. Pte. Ltd., Hackensack, NJ, 2009.

\bibitem{Nishioka1996}
K.~Nishioka.
\newblock Monopoles and dipoles in biharmonic pseudo-process.
\newblock {\em Proc. Japan Acad. Ser. A Math. Sci.}, 72(3):47--50, 1996.

\bibitem{Nishioka2001}
K.~Nishioka.
\newblock Boundary conditions for one-dimensional biharmonic pseudo process.
\newblock {\em Electron. J. Probab.}, 6:no.\ 13, 27 pp. (electronic), 2001.

\bibitem{Orsingher1999}
E.~Orsingher and X.~Zhao.
\newblock Iterated processes and their applications to higher order
  differential equations.
\newblock {\em Acta Math. Sin. (Engl. Ser.)}, 15(2):173--180, 1999.

\bibitem{Sinestrari1976}
E.~Sinestrari.
\newblock Accretive differential operators.
\newblock {\em Boll. Un. Mat. Ital. B (5)}, 13(1):19--31, 1976.

\bibitem{Thomas2001}
E.~G.~F. Thomas.
\newblock Projective limits of complex measures and martingale convergence.
\newblock {\em Probab. Theory Related Fields}, 119(4):579--588, 2001.

\end{thebibliography}
\end{document}